\documentclass[preprint,3p,12pt,times]{elsarticle}

\usepackage{amsmath}
\usepackage{amssymb}
\usepackage{mathrsfs}
\def\imnum{\mathrm{i}\,}
\def\divv{\mathrm{d}}
\def\diff{\,\mathrm{d}}

\newdefinition{defn}{Definition}[section]
\newdefinition{rem}{Remark}
\newdefinition{exam}{Example}
\newtheorem{thm}{Theorem}[section]
\newtheorem{lem}[thm]{Lemma}

\newproof{proof}{Proof}

\renewcommand{\Re}{\mathop{\mathrm{Re}}\nolimits}
\renewcommand{\Im}{\mathop{\mathrm{Im}}\nolimits}

\DeclareMathOperator{\rme}{e}
\DeclareMathOperator{\arcsinh}{arcsinh}

\DeclareMathOperator{\sech}{sech}
\DeclareMathOperator{\Si}{Si}

\DeclareMathOperator{\Order}{O}

\DeclareMathOperator{\Bfunc}{B}

\renewcommand{\pi}{\piup}
\newcommand{\Nto}{N_{\text{total}}}

\newcommand{\textDE}{\text{\tiny{\rm{DE}}}}

\newcommand{\DEt}{\psi_{\textDE}}

\newcommand{\textDEzero}{\text{\tiny{\rm{(0,1)}}}}
\newcommand{\DEtZero}{\psi^{\textDEzero}_{\textDE}}
\newcommand{\DEtInv}{\DEt^{-1}}
\newcommand{\DEtDiv}{\DEt'}

\newcommand{\domD}{\mathscr{D}}

\numberwithin{equation}{section}

\journal{Elsevier}

\begin{document}

\begin{frontmatter}



\title{Explicit error bound for modified numerical iterated integration
by means of Sinc methods\tnoteref{thankslabel}}
\tnotetext[thankslabel]{This work was supported by Grant-in-Aid for
Young Scientists (B) Number 24760060.}


\author{Tomoaki Okayama}

\address{Graduate School of Economics, Hitotsubashi University,
2-1 Naka, Kunitachi, Tokyo 186-8601, Japan}
\ead{tokayama@econ.hit-u.ac.jp}

\begin{abstract}
This paper reinforces numerical iterated integration
developed by Muhammad--Mori in the following two points:
1)
the approximation formula is modified
so that it can achieve a better convergence rate
in more general cases, and
2) explicit error bound is given
in a computable form for the modified formula.
The formula works quite efficiently,
especially if the integrand is of a product type.
Numerical examples that confirm it are also presented.
\end{abstract}

\begin{keyword}
Sinc quadrature
\sep Sinc indefinite integration
\sep repeated integral
\sep verified numerical integration
\sep double-exponential transformation
\MSC 65D30 \sep 65D32 \sep 65G99
\end{keyword}

\end{frontmatter}



\setlength{\abovedisplayskip}{5pt}
\setlength{\belowdisplayskip}{5pt}


\section{Introduction}
The concern of this paper is efficient approximation
of a two-dimensional iterated integral
\begin{equation}
I = \int_a^b\left(\int_A^{q(x)}
f(x,y)
\diff y\right)\diff x,
\label{eq:target-repeat-int}
\end{equation}
with giving its \emph{strict} error bound.
Here,
$q(x)$ is a monotone function that
may have derivative singularity at the endpoints of $[a,\,b]$,
and the integrand $f(x,y)$ also may have singularity on the boundary of the
square region $[a,\,b]\times [A,\,B]$
(
see also Figs.~\ref{fig:monotone-increase} and~\ref{fig:monotone-decrease}).
In this case,
a Cartesian product rule of
a well known one-dimensional quadrature formula
(such as the Gaussian formula
and the Clenshaw--Curtis formula)
does not work properly, or at least
its mathematically-rigorous error bound is quite difficult to obtain,
because such formulas require the analyticity
of the integrand
in a neighbourhood of the boundary~\cite{eiermann89:_autom}.
\begin{figure}[htbp]
\begin{center}
\begin{minipage}{0.4\linewidth}
\begin{center}
 \includegraphics[scale=.75]{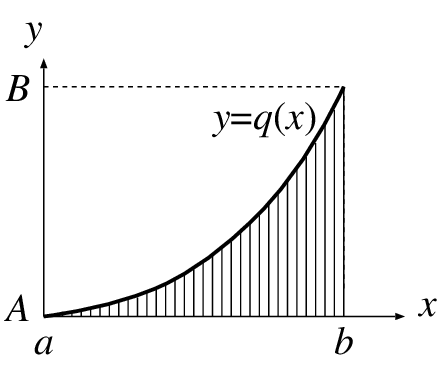}
 \caption{The domain of integration~\eqref{eq:target-repeat-int} when $q'(x)\geq 0$.}
 \label{fig:monotone-increase}
\end{center}
\end{minipage}
\begin{minipage}{0.05\linewidth}
\mbox{ }
\end{minipage}
\begin{minipage}{0.4\linewidth}
\begin{center}
 \includegraphics[scale=.75]{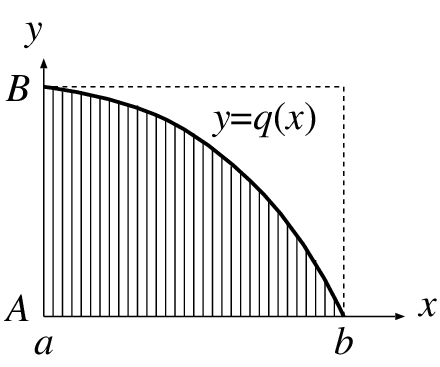}
 \caption{The domain of integration~\eqref{eq:target-repeat-int} when $q'(x)\leq 0$.}
 \label{fig:monotone-decrease}
\end{center}
\end{minipage}
\end{center}
\end{figure}

Promising quadrature formulas
that does not require the analyticity at the endpoints
may include the tanh formula~\cite{schwartz69:_numer},
the IMT formula~\cite{iri70:_certain,iri87:_certain},
and the double-exponential formula~\cite{takahasi74:_doubl},
which enjoy \emph{exponential convergence}
whether the integrand has such singularity or not.
Actually, based on the IMT formula,
an automatic integration algorithm for~\eqref{eq:target-repeat-int}
was developed~\cite{robinson81:_algor}.
Further improved version was developed as
d2lri~\cite{hill99:_d2lri}
and r2d2lri~\cite{robinson02:_algor},
where the lattice rule is employed with
the IMT transformation~\cite{iri70:_certain,iri87:_certain}
or the Sidi transformation~\cite{sidi93:_new_variab,sidi06:_exten}.
As a related study,
based on the double-exponential formula,
an automatic integration algorithm over a sphere was
developed~\cite{roose81:_autom},
which also intended to deal with such integrand singularity.
The efficiency of those algorithms are also suggested
by their numerical experiments.

From a mathematical viewpoint, however,
those algorithms do not guarantee the accuracy of the approximation
in reality.
In order to estimate the error (for giving a stop criterion),
Robinson and de Doncker~\cite{robinson81:_algor} considered
the sequence of the number of function evaluation points $\{N_m\}_{m}$ and
that of approximation values $\{I_{N_m}\}$,
and made the important assumption:
\begin{equation}
D_{N_m}:= |I_{N_m} - I_{N_{m-1}}| \simeq |I - I_{N_{m-1}}|,
\label{simeq:robinson-assump}
\end{equation}
which enables the error estimation
$|I- I_{N_m}|\simeq D_{N_m}^2/D_{N_{m-1}}$.
Similar approach was taken in the studies
described above~\cite{hill99:_d2lri,robinson02:_algor,roose81:_autom}.
The problem here is that
it is quite difficult to guarantee the validity
of~\eqref{simeq:robinson-assump},
although it had been widely accepted as a realistic practical assumption
for constructing automatic quadrature routines in that period.
The recent trend is that
the approximation error is bounded by a \emph{strict} inequality
(instead of estimation `$\simeq$') as
\begin{equation*}
|I - I_N|\leq E_N,
\end{equation*}
where $E_N$ is given in a \emph{computable} form
(see, for example, Petras~\cite{petras07:_princ}).
Such an explicit error bound is desired
for constructing a more reliable,
\emph{verified numerical integration} routine.
In addition to the mathematical rigorousness,
such a bound gives us another advantage:
the sufficient number of $N$ for the required precision, say $N_0$,
can be known without generating the sequence $\{I_{N}\}$.
This means low computational cost,
since we do not have to compute for any $N$ with $N<N_0$
(and of course $N>N_0$).

The objective of this study is
to give such an explicit error bound
for the numerical integration method developed by
Muhammad--Mori~\cite{muhammad05:_iterat}.
Their method is based on the Sinc
methods~\cite{stenger93:_numer,stenger00:_summar}
combined with
double-exponential transformation~\cite{mori01:_doubl,takahasi74:_doubl},
and it has the following two features:
\begin{enumerate}
 \item it has beautiful \emph{exponential accuracy}
even if $f(x,y)$ or $q(x)$ has boundary singularity, and
 \item it employs \emph{indefinite integration formula}
instead of quadrature formula for the inner integral.
\end{enumerate}
The first point is the same feature as the studies
above~\cite{robinson81:_algor,roose81:_autom},
but the second point is a unique one.
If a standard quadrature rule is employed
to approximate the inner integral,
the weight $w_j$ and quadrature node $y_j$ should be adjusted
depending on $x$ as
\[
 \int_A^{q(x)}f(x, y)\diff y
\approx \sum_{j}w_j(x) f(x,y_j(x)),
\]
whereas in the case of an \emph{indefinite integration formula},
$y_j$ is fixed (independent of $x$) as
\[
 \int_A^{q(x)}f(x, y)\diff y
\approx \sum_{j}w_j(x) f(x,y_j).
\]
This independency on $x$ is quite useful to check
mathematical assumptions on the integrand $f(x,y)$
for the exponential accuracy.
Furthermore, as a special case,
when the integrand is of a product type:
$f(x,y)=X(x)Y(y)$,
the number of function evaluation
to approximate~\eqref{eq:target-repeat-int}
is drastically dropped from $\Order(n\times n)$ to $\Order(n+n)$,
where $n$ denotes the number of the terms of $\sum$
(it is also emphasized in the original paper~\cite{muhammad05:_iterat}).

However,
rigorous error analysis is not given for the formula,
and there is room for improvement in the convergence rate.
Moreover,
it cannot handle the case $q'(x)\leq 0$
(only the case $q'(x)\geq 0$ is considered).
In order to reinforce their formula, this study contributes in
the following points:
\begin{enumerate}
\setcounter{enumi}{2}
\item their formula is modified
so that it can achieve a better convergence rate
in both cases (i.e., the case $q'(x)\geq 0$ and $q'(x)\leq 0$),
and
\item a rigorous, explicit error bound is given for the modified formula.
\end{enumerate}
From the error bound in the latter point,
we can see that the convergence rate of the formula is
generally $\Order(\exp(-c \sqrt{n}/\log(\gamma \sqrt{n})))$,
and if $f(x,y)=X(x)Y(y)$,
it becomes $\Order(\exp(-c' n/\log(\gamma' n)))$.

The remainder of this paper is organized as follows.
In Section~\ref{sec:muhammad_mori},
after the review of basic formulas of Sinc methods,
Muhammad-Mori's original formula~\cite{muhammad05:_iterat}
is described.
Then, the formula is modified
in Section~\ref{sec:approx_formula},
and its explicit error bound is also presented.
Its proof is given in Section~\ref{sec:proofs}.
Numerical examples are shown in Section~\ref{sec:numer_exam}.
Section~\ref{sec:conclusion} is devoted to conclusion.


\section{Review of Muhammad--Mori's approximation formula}
\label{sec:muhammad_mori}

In this section,
the approximation formula for~\eqref{eq:target-repeat-int}
derived by
Muhammad--Mori~\cite{muhammad05:_iterat} is described.
The idea is to use ``Sinc quadrature''
for the outer integral,
and to use ``Sinc indefinite integration''
for the inner integral. Those two approximation formulas are explained first.

\subsection{Sinc quadrature and Sinc indefinite integration combined with the DE transformation}

The Sinc quadrature and Sinc indefinite integration
are approximation formulas for definite integration
and indefinite integration, respectively, expressed as
\begin{align}
\int_{-\infty}^{\infty} G(\xi)\diff \xi
&\approx \tilde{h} \sum_{i=-M_{-}}^{M_{+}} G(i \tilde{h}),
\label{approx:Sinc-quad}\\
\int_{-\infty}^{\xi} G(\eta)\diff\eta
&\approx \sum_{j=-N_{-}}^{N_{+}} G(j h) J(j, h)(\xi),
\quad\xi\in\mathbb{R},\label{approx:Sinc-indef}
\end{align}
where $J(j,h)(\xi)$ is defined by using
the so-called sine integral $\Si(x)=\int_0^x\{(\sin\sigma)/\sigma\}\diff\sigma$
as
\[
 J(j,h)(\xi)=h\left\{\frac{1}{2}+\frac{1}{\pi}\Si[\pi(\xi/h - j)]\right\}.
\]
Although the formulas~\eqref{approx:Sinc-quad}
and~\eqref{approx:Sinc-indef} are
approximations on the whole real line $\mathbb{R}$,
those can be used on the finite interval $(a,\,b)$ as well,
by using the Double-Exponential (DE) transformation
\begin{align*}
x = \DEt(\xi)
&=\frac{b-a}{2}\tanh\left(\frac{\pi}{2}\sinh\xi\right)+\frac{b+a}{2}.
\end{align*}
Since $\DEt:\mathbb{R}\to (a,\,b)$,
we can apply the formulas~\eqref{approx:Sinc-quad}
and~\eqref{approx:Sinc-indef}
in the case of finite intervals
combining the DE transformation as
\begin{align}
\int_{a}^{b}g(x)\diff x
&=\int_{-\infty}^{\infty} g(\DEt(\xi))\DEtDiv(\xi)\diff \xi
\approx \tilde{h}
 \sum_{i=-M_{-}}^{M_{+}} g(\DEt(i \tilde{h}))\DEtDiv(i \tilde{h}),
\label{approx:DE-Sinc-quad}\\
\int_{a}^{x}g(y)\diff y
&=\int_{-\infty}^{\DEtInv(x)} g(\DEt(\eta))\DEtDiv(\eta)\diff\eta
\approx \sum_{j=-N_{-}}^{N_{+}} g(\DEt(j h))\DEtDiv(j h)
J(j, h)(\DEtInv(x)),
\quad x\in (a,\,b),
\label{approx:DE-Sinc-indef}
\end{align}
which are called the ``DE-Sinc quadrature''
and the ``DE-Sinc indefinite integration,''
proposed by
Takahasi--Mori~\cite{takahasi74:_doubl}
and
Muhammad--Mori~\cite{muhammad03:_doubl},
respectively.

\subsection{Muhammad--Mori's approximation formula}

Let the domain of integration~\eqref{eq:target-repeat-int}
be as in Fig.~\ref{fig:monotone-increase},
i.e., $q(a)=A$, $q(b)=B$, and $q'(x)\geq 0$.
Using the monotonicity of $q(x)$,
Muhammad--Mori~\cite{muhammad05:_iterat}
rewrote the given integral $I$ by applying $y=q(s)$ as
\begin{equation}
 I=\int_a^b\left(\int_A^{q(x)}
f(x,y)
\diff y\right)\diff x
=\int_a^b\left(
\int_a^x
f(x,q(s))q'(s)\diff s
\right)\diff x.
\label{eq:MM-pretransform}
\end{equation}
Note that $s\in (a,\,b)$ (i.e., not $(A,\,B)$).
Then, they applied~\eqref{approx:DE-Sinc-quad}
and~\eqref{approx:DE-Sinc-indef},
with taking $\tilde{h}=h$, $M_{-}=M_{+}=m$, and $N_{-}=N_{+}=n$ for simplicity,
as follows:
\begin{align*}
I&\approx h\sum_{i=-m}^{m}\DEtDiv(ih)
\left(
\int_a^{\DEt(ih)}f(\DEt(ih),q(s))q'(s)\diff s
\right)\\
&\approx h\sum_{i=-m}^{m}\DEtDiv(ih)
\left\{
\sum_{j=-n}^n f(\DEt(ih),q(\DEt(jh)))q'(\DEt(jh))\DEtDiv(jh)J(j,h)(ih)
\right\}.
\end{align*}
If we introduce
$x_i=\DEt(ih)$,
$w_j=\pi\cosh(jh) \sech^2(\pi\sinh(jh)/2) /4$,
and $\sigma_k=\Si[\pi k]/\pi$, which can be prepared
in prior to computation
(see also a value
table for $\sigma_k$~\cite[Table~1.10.1]{stenger93:_numer}),
the formula is rewritten as
\begin{equation}
I\approx
(b-a)^2h^2\sum_{i=-m}^m w_i
\left\{
\sum_{j=-n}^n f(x_i, q(x_j)) q'(x_j) w_j
\left(\frac{1}{2}+\sigma_{i-j}\right)
\right\}.
\label{approx:MM-formula}
\end{equation}
The total number of function evaluations,
say $\Nto$,
of this formula is $\Nto=(2m+1)\times(2n+1)$.
As a special case,
if the integrand is of a product type: $f(x,y)=X(x)Y(y)$,
%
the formula is rewritten as
\begin{align}
I\approx
(b-a)^2h^2\sum_{i=-m}^m U(i)
\left\{\sum_{j=-n}^n V(j)
 \left(\frac{1}{2}+\sigma_{i-j}\right)
\right\},
\label{approx:MM-formula-product}
\end{align}
where $U(i)=X(x_i)w_i$ and
$V(j)=Y(q(x_j))q'(x_j)w_j$.
In this case,
we can see that
$\Nto=(2m+1)+(2n+1)$,
which is significantly smaller than $(2m+1)\times(2n+1)$.

They~\cite{muhammad05:_iterat} also roughly discussed the error rate
of the formula~\eqref{approx:MM-formula} as follows.
Let $\domD_d$ be a strip domain defined by
$\domD_d=\{\zeta\in\mathbb{C}:|\Im\zeta|<d\}$ for $d>0$.
Assume that
the integrand $g$ in~\eqref{approx:DE-Sinc-quad}
and~\eqref{approx:DE-Sinc-indef}
is analytic on $\DEt(\domD_d)$
(which means $g(\DEt(\cdot))$ is analytic on $\domD_d$),
and further assume that $g(x)$ behaves $\Order(((x-a)(b-x))^{\nu-1})$ ($\nu>0$)
as $x\to a$ and $x\to b$.
Under the assumptions with some additional mild conditions,
it is known that
the approximation~\eqref{approx:DE-Sinc-quad}
converges with
$\Order(\rme^{-2\pi d/h})$,
and the approximation~\eqref{approx:DE-Sinc-indef}
converges with
$\Order(h\rme^{-\pi d/h})$,
by taking $h=\tilde{h}$ and
\[
 M_{+}=M_{-}=m=
\left\lceil \frac{1}{h}\log\left(\frac{4 d}{(\nu-\epsilon) h}\right)
\right\rceil,
\quad
 N_{+}=N_{-}=n
=\left\lceil \frac{1}{h}\log\left(\frac{2 d}{(\nu-\epsilon) h}\right)
\right\rceil,
\]
where $\epsilon$ is an arbitrary small positive number.
Therefore, if the same assumptions are satisfied
for both approximations in~\eqref{approx:MM-formula},
it enjoys exponential accuracy:
$\Order(h\rme^{-\pi d/h})$.
Since
$m\simeq n \simeq \sqrt{\Nto/4}$
and $h\simeq \log(cn)/n$
(where $c=2d/(\nu-\epsilon)$),
this can be interpreted in terms of $\Nto$ as
\begin{equation}
 \Order\left(\frac{\log(c\sqrt{\Nto/4})}{\sqrt{\Nto/4}}
\exp\left[
\frac{-\pi d \sqrt{\Nto/4}}{\log(c \sqrt{\Nto/4})}
\right]
\right).
\label{order:MM-orig}
\end{equation}
If the integrand is of a product type,
since $m\simeq n\simeq \Nto/4$,
it becomes
\begin{equation}
 \Order\left(\frac{\log(c\Nto/4)}{\Nto/4}
\exp\left[\frac{-\pi d (\Nto/4)}{\log(c \Nto/4)}
\right]
\right).
\label{order:MM-product}
\end{equation}

Although the convergence rate was roughly discussed as above,
the quantity of the approximation error cannot be
obtained because rigorous error bound was not given.
Moreover,
the case $q'(x)\leq 0$
(cf. Fig.~\ref{fig:monotone-decrease})
is not considered.
This situation will be improved in the next section.




\section{Main results: modified approximation formula and its explicit error bound}
\label{sec:approx_formula}

This section
is devoted to a description of a new approximation formula
and its error bound.
The proof of the error bound is given in Section~\ref{sec:proofs}.

\subsection{Modified approximation formula}

In the approximations~\eqref{approx:DE-Sinc-quad}
and~\eqref{approx:DE-Sinc-indef},
Muhammad--Mori~\cite{muhammad05:_iterat}
set the mesh size as $\tilde{h}=h$ for simplicity,
but here, $\tilde{h}$ is selected as $\tilde{h}=2h$.
Furthermore, both $M_{-}=M_{+}$ and $N_{-}=N_{+}$ are \emph{not} assumed.
Then, after applying $y=q(s)$ as in~\eqref{eq:MM-pretransform},
the modified formula is derived as
\begin{align*}
I
&\approx 2 h\sum_{i=-M_{-}}^{M_{+}} \DEtDiv(2ih)
\left(
\int_a^{\DEt(2ih)} f(\DEt(2ih), q(s))q'(s)\diff s
\right) \\
&\approx
2h\sum_{i=-M_{-}}^{M_{+}} \DEtDiv(2ih)
\left\{
\sum_{j=-N_{-}}^{N_{+}} f(\DEt(2ih), q(\DEt(jh)))
q'(\DEt(jh))
\DEtDiv(jh)
J(j,h)(2ih)
\right\},
\end{align*}
which can be rewritten as
\begin{equation}
I\approx I_{\textDE}^{\text{\rm{inc}}}(h):=
2(b-a)^2h^2\sum_{i=-M_{-}}^{M_{+}} w_{2i}
\left\{
\sum_{i=-N_{-}}^{N_{+}} f(x_{2i}, q(x_j)) q'(x_j) w_j
\left(\frac{1}{2}+\sigma_{2i-j}\right)
\right\}.
\label{approx:O-DE-formula-inc}
\end{equation}
The positive integers
$M_{\pm}$ and $N_{\pm}$ are also selected depending on $h$,
which is explained in the subsequent theorem
that states the error bound.

The formula~\eqref{approx:O-DE-formula-inc}
is derived in the case $q'(x)\geq 0$ (cf. Fig.~\ref{fig:monotone-increase}),
but
in the case $q'(x)\leq 0$ (cf. Fig.~\ref{fig:monotone-decrease}) as well,
we can derive the similar formula as follows.
First, applying $y=q(s)$, we have
\begin{align*}
 I=\int_a^b\left(\int_A^{q(x)}f(x,y)\diff y\right)\diff x
&=\int_a^b\left(\int_x^b
f(x,q(s))\{-q'(s)\}\diff s
\right)\diff x\\
&=\int_a^b\left(\int_a^b
f(x,q(s))\{-q'(s)\}\diff s
-\int_a^x
f(x,q(s))\{-q'(s)\}\diff s
\right)\diff x.
\end{align*}
Then, apply~\eqref{approx:DE-Sinc-quad}
and~\eqref{approx:DE-Sinc-indef} to obtain
\begin{equation*}
I\approx
2h\sum_{i=-M_{-}}^{M_{+}} \DEtDiv(2ih)
\left\{
\sum_{j=-N_{-}}^{N_{+}} f(\DEt(2ih), q(\DEt(jh)))
\{-q'(\DEt(jh))\}
\DEtDiv(jh)
\left(h - J(j,h)(2ih)\right)
\right\}.
\end{equation*}
Here, $\lim_{\xi\to\infty}J(j,h)(\xi)=h$ is used.
This approximation can be rewritten as
\begin{equation}
I\approx I_{\textDE}^{\text{\rm{dec}}}(h):=
2(b-a)^2h^2\sum_{i=-M_{-}}^{M_{+}} w_{2i}
\left\{
\sum_{i=-N_{-}}^{N_{+}} f(x_{2i}, q(x_j)) \{-q'(x_j)\} w_j
\left(\frac{1}{2}-\sigma_{2i-j}\right)
\right\}.
\label{approx:O-DE-formula-dec}
\end{equation}
The formulas~\eqref{approx:O-DE-formula-inc}
and~\eqref{approx:O-DE-formula-dec}
inherit the advantage of
Muhammad--Mori's one in the sense that
$\Nto=(M_{-}+M_{+}+1)\times (N_{-}+N_{+}+1)$ in general,
but if the integrand is of a product type: $f(x,y)=X(x)Y(y)$,
it becomes $\Nto=(M_{-}+M_{+}+1)+ (N_{-}+N_{+}+1)$,
which is easily confirmed
by rewriting it
in the same way as~\eqref{approx:MM-formula-product}.
Furthermore, it also inherits (or even enhances)
the \emph{exponential} accuracy,
which is described next.


\subsection{Explicit error bound of the modified formula}

For positive constants $\kappa$, $\lambda$
and $d$ with $0<d<\pi/2$,
let us define $c_{\kappa,\lambda,d}$ as
\[
 c_{\kappa,\lambda,d}
=\frac{1}{\cos^{\kappa+\lambda}(\frac{\pi}{2}\sin d)\cos d},
\]
and define $\rho_{\kappa}$ as
\begin{equation*}
\rho_{\kappa}=
\begin{cases}
\arcsinh\left(
\frac{\sqrt{1+\sqrt{1-(2\pi \kappa)^2}}}{2 \pi \kappa}
\right)
& (0<\kappa<1/(2\pi)),\\
\arcsinh(1)&(1/(2\pi)\leq \kappa).
\end{cases}
\end{equation*}
Then, the errors of $I_{\textDE}^{\text{\rm{inc}}}(h)$
and $I_{\textDE}^{\text{\rm{dec}}}(h)$
are estimated as stated below.

\begin{thm}
\label{thm:main-result}
Let $\alpha$, $\beta$, $\gamma$, $\delta$, and $K$
be positive constants, and $d$ be a constant with $0<d<\pi/2$.
Assume the following conditions:
\begin{enumerate}
 \item $q$ is analytic and bounded in $\DEt(\domD_d)$,
 \item $f(\cdot,q(w))$ is analytic in $\DEt(\domD_d)$
for all $w\in\DEt(\domD_d)$,
 \item $f(z,q(\cdot))$ is analytic in $\DEt(\domD_d)$
for all $z\in\DEt(\domD_d)$,
 \item it holds for all $z\in\DEt(\domD_d)$ and $w\in\DEt(\domD_d)$ that
\begin{equation}
 |f(z,q(w))q'(w)|
\leq K |z-a|^{\alpha-1}|b-z|^{\beta-1}|w-a|^{\gamma-1}|b-w|^{\delta-1}.
\label{leq:bound-f-q}
\end{equation}
\end{enumerate}
Let ${\mu}=\min\{\alpha,\beta\}$,
$\overline{\mu}=\max\{\alpha,\beta\}$,
${\nu}=\min\{\gamma,\delta\}$,
$\overline{\nu}=\max\{\gamma,\delta\}$,
let $\tilde{h}=2h$,
let $n$ and $m$ be positive integers defined by
\begin{equation}
n=\left\lceil
\frac{1}{h}\log\left(\frac{2d}{\nu h}\right)
\right\rceil,
\quad
m=\left\lceil
\frac{1}{2}\left\{
n + \frac{1}{h}\log\left(\frac{\mu}{\nu}\right)
\right\}
\right\rceil,
\label{eq:m-and-n-def}
\end{equation}
and let $M_{-}$ and $M_{+}$ be positive integers defined by
\begin{equation}
\begin{cases}
M_{-}=m,\quad M_{+}=m-\lfloor\log(\beta/\alpha)/\tilde{h}\rfloor
 & \quad (\text{if}\,\,\,\mu = \alpha),\\
M_{+}=m,\quad M_{-}=m-\lfloor\log(\alpha/\beta)/\tilde{h}\rfloor
 & \quad (\text{if}\,\,\,\mu = \beta),
\end{cases}
\label{DE-Sinc-func-approx-M-def}
\end{equation}
and let $N_{-}$ and $N_{+}$ be positive integers defined by
\begin{equation}
\begin{cases}
N_{-}=n,\quad N_{+}=n-\lfloor\log(\delta/\gamma)/(h)\rfloor
 & \quad (\text{if}\,\,\,\nu = \gamma),\\
N_{+}=n,\quad N_{-}=n-\lfloor\log(\gamma/\delta)/(h)\rfloor
 & \quad (\text{if}\,\,\,\nu = \delta),
\end{cases}
\label{DE-Sinc-func-approx-N-def}
\end{equation}
and let $h$ $(>0)$ be taken sufficiently small so that
\begin{equation*}
 M_{-}\tilde{h}\geq \rho_{\alpha},\quad
 M_{+}\tilde{h}\geq \rho_{\beta},\quad
 N_{-}h\geq \rho_{\gamma},\quad
 N_{+}h\geq \rho_{\delta}
\end{equation*}
are all satisfied. Then, if $q'(x)\geq 0$, it holds that
\begin{align}
&|I - I_{\textDE}^{\text{\rm{inc}}}(h)|\nonumber\\&\leq
\left[
\frac{\Bfunc(\gamma,\delta)c_{\gamma,\delta,d}}{\mu}
\left\{
\rme^{\frac{\pi}{2}\overline{\mu}}
+\frac{2c_{\alpha,\beta,d}}{1-\rme^{-\pi d/h}}
\right\}
+\frac{1}{\nu}
\left\{\Bfunc(\alpha,\beta)+\frac{4 c_{\alpha,\beta,d}}{\mu}
\frac{\rme^{-\pi d/h}}{1-\rme^{-\pi d /h}}
\right\}
\left\{
1.1\rme^{\frac{\pi}{2}\overline{\nu}}
+\frac{hc_{\gamma,\delta,d}}{d(1-\rme^{-2\pi d/h})}
\right\}
\right]\nonumber\\
&\quad\times
2K (b-a)^{\alpha+\beta+\gamma+\delta-2} \rme^{-\pi d/h},
\label{leq:main-result-error-bound}
\end{align}
where $\Bfunc(\kappa,\lambda)$ is the beta function.
If $q'(x)\leq 0$,
$|I - I_{\textDE}^{\text{\rm{dec}}}(h)|$
is bounded by the same term on the right hand side
of~\eqref{leq:main-result-error-bound}.
\end{thm}

The convergence rate of~\eqref{leq:main-result-error-bound}
is $\Order(\rme^{-\pi d/h})$,
which can be interpreted in terms of $\Nto$ as follows.
Since $n\simeq N_{-}\simeq N_{+}$ and
$m\simeq M_{-}\simeq M_{+}\simeq (n/2)$,
we can see $\Nto\simeq ((n/2)+(n/2)+1)(n+n+1)\simeq 2n^2$.
From this and $h\simeq \log(c'n)/n$
(where $c'=2d/\nu$),
the convergence rate of the modified formula is
\[
 \Order\left(
\exp\left[
\frac{-\pi d \sqrt{\Nto/2}}{\log(c' \sqrt{\Nto/2})}
\right]
\right).
\]
This rate is better than
Muhammad--Mori's one~\eqref{order:MM-orig}.
If the integrand is of a product type: $f(x,y)=X(x)Y(y)$,
it becomes
\[
  \Order\left(
\exp\left[\frac{-\pi d (\Nto/3)}{\log(c \Nto/3)}
\right]
\right),
\]
since $\Nto\simeq ((n/2)+(n/2)+1)+(n+n+1)\simeq 3n$
in this case.
This rate is also better than
Muhammad--Mori's one~\eqref{order:MM-product}.

\begin{rem}
The inequality~\eqref{leq:main-result-error-bound}
states the bound of the \emph{absolute} error, say $E^{\text{\rm{abs}}}(h)$.
If necessary, the bound of the \emph{relative} error
$E^{\text{\rm{rel}}}(h)$
is also obtained as follows:
\[
E^{\text{\rm{rel}}}(h)=
 \frac{|I - I_{\textDE}^{\text{\rm{inc}}}(h)|}{|I|}
\leq  \frac{|E^{\text{\rm{abs}}}(h)|}{|I|}
\leq \frac{|E^{\text{\rm{abs}}}(h)|}{||I_{\textDE}^{\text{\rm{inc}}}(h)| - E^{\text{\rm{abs}}}(h)|}.
\]
\end{rem}


\section{Numerical examples}
\label{sec:numer_exam}

In this section, numerical results of
Muhammad-Mori's original formula~\cite{muhammad05:_iterat}
and modified formula are presented.
The results of an existing library: r2d2lri~\cite{robinson02:_algor},
which can properly handle boundary singularity in $q(x)$ and $f(x,y)$,
are also shown.
The computation was done on Mac OS X 10.6,
Mac Pro two 2.93~GHz 6-Core Intel Xeon with 32 GB DDR3 ECC SDRAM.
The computation programs were implemented
in C/C++ with double-precision floating-point arithmetic,
and compiled by GCC 4.0.1 with no optimization.
The following three examples were conducted.

\begin{exam}[{\rm The integrand and boundary function are smooth~\cite[Example~2]{muhammad05:_iterat}}]
\label{exam:smooth}
\[
 \int_0^{\sqrt{2}}
\left(\int_0^{x^2/2}\frac{\divv y}{x+y+(1/2)}\right)\diff x
=-\left(\sqrt{2}+\frac{1}{2}\right)
\log\left(1+2\sqrt{2}\right) + 2\left(1+\sqrt{2}\right)
\log\left(1+\sqrt{2}\right)-\sqrt{2}.
\]
\end{exam}
\begin{exam}[{\rm Derivative singularity exists in the integrand
and boundary function~\cite[Example~1]{muhammad05:_iterat}}]
\label{exam:deriv-singular}
\[
 \int_0^1\left(\int_0^{\sqrt{1-(1-x)^2}}\sqrt{1-y^2}\diff y\right)\diff x
=\frac{2}{3}.
\]
\end{exam}
\begin{exam}[{\rm The integrand is weakly singular at the origin~\cite[Example~27]{hill99:_d2lri}}]
\label{exam:weak-singular}
\[
 \int_0^1\left(\int_0^{1-x}\frac{\divv y}{\sqrt{xy}}\right)\diff x
=\pi.
\]
\end{exam}

In the case of Example~\ref{exam:smooth},
the assumptions in Theorem~\ref{thm:main-result}
are satisfied with $\alpha=\beta=\delta=1$, $\gamma=2$,
$d=\log(2)$, and $K=16.6$.
The results are shown in Figs.~\ref{fig:smooth-N} and~\ref{fig:smooth-t}.
In both figures, error bound (say $\tilde{E}^{\text{rel}}(h)$)
given by Theorem~\ref{thm:main-result}
surely includes the observed relative error $E^{\text{rel}}(h)$
in the form $E^{\text{rel}}(h) \leq \tilde{E}^{\text{rel}}(h)$,
which is also true in all the subsequent examples
(note that such error bound is not given for Muhammad--Mori's original formula).
In view of the performance,
r2d2lri is better than original/modified formulas,
but its error estimate just claims
$E^{\text{rel}}(h) \approx \tilde{E}^{\text{rel}}(h)$,
and does not guarantee
$E^{\text{rel}}(h) \leq \tilde{E}^{\text{rel}}(h)$ mathematically.

In the case of Example~\ref{exam:deriv-singular},
the assumptions in Theorem~\ref{thm:main-result}
are satisfied with $\alpha=\beta=1$, $\gamma=1/2$, $\delta=3$,
$d=1$, and $K=1.63$.
The results are shown in Figs.~\ref{fig:deriv-singular-N}
and~\ref{fig:deriv-singular-t}.
In this case, the convergence of the
original/modified formulas is incredibly fast
compared to r2d2lri.
This is because the integrand is
of a product type: $f(x,y)=X(x)Y(y)$.

The integrand of Example~\ref{exam:weak-singular}
is also of a product type.
In this example,
the assumptions in Theorem~\ref{thm:main-result}
are satisfied with $\alpha=\delta=1/2$, $\beta=\gamma=1$,
$d=4/3$, and $K=1$.
The results are shown in Figs.~\ref{fig:weak-singular-N}
and~\ref{fig:weak-singular-t}.
In this case, the performance of r2d2lri
is much worse than that in Example~\ref{exam:deriv-singular},
which seems to be due to the singularity of the integrand.
In contrast, the modified formula attains the similar
convergence rate to that in Example~\ref{exam:deriv-singular}.
Muhammad--Mori's original formula cannot be used in this case
since $q(x)=1-x$ does not satisfy $q'(x)\geq 0$.

\begin{figure}[htbp]
\begin{center}
\begin{minipage}{0.45\linewidth}
\includegraphics[width=\linewidth]{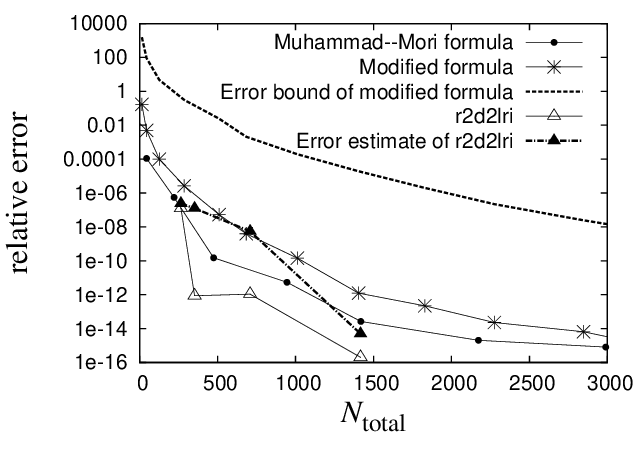}
\caption{Relative error with respect to $N_{\text{total}}$ in Example~\ref{exam:smooth}.}
\label{fig:smooth-N}
\end{minipage}
\begin{minipage}{0.05\linewidth}
\mbox{ }
\end{minipage}
\begin{minipage}{0.45\linewidth}
\includegraphics[width=\linewidth]{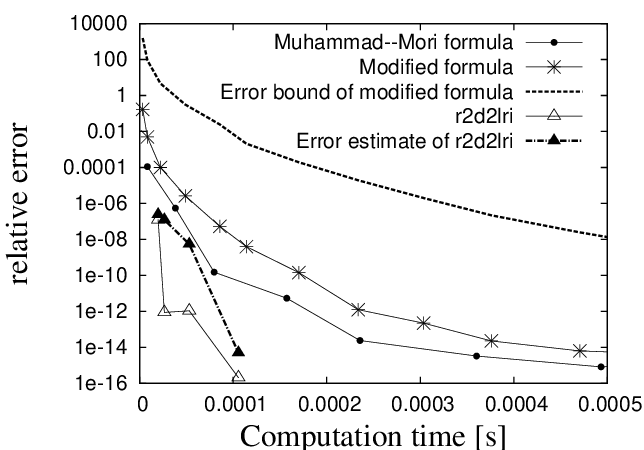}
\caption{Relative error with respect to computation time in Example~\ref{exam:smooth}.}
\label{fig:smooth-t}
\end{minipage}
\end{center}

\begin{center}
\begin{minipage}{0.45\linewidth}
\includegraphics[width=\linewidth]{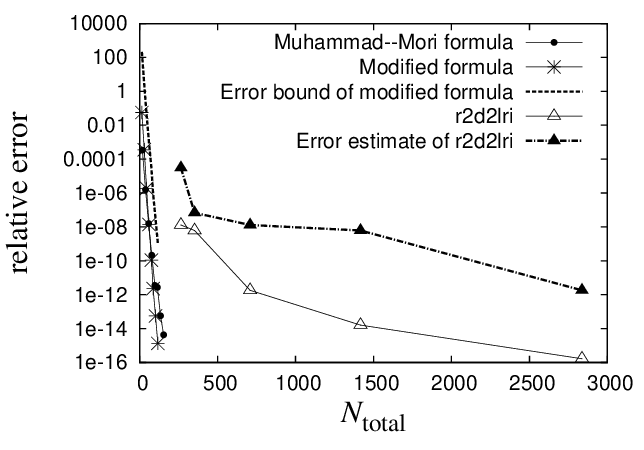}
\caption{Relative error with respect to $N_{\text{total}}$ in Example~\ref{exam:deriv-singular}.}
\label{fig:deriv-singular-N}
\end{minipage}
\begin{minipage}{0.05\linewidth}
\mbox{ }
\end{minipage}
\begin{minipage}{0.43\linewidth}
\includegraphics[width=\linewidth]{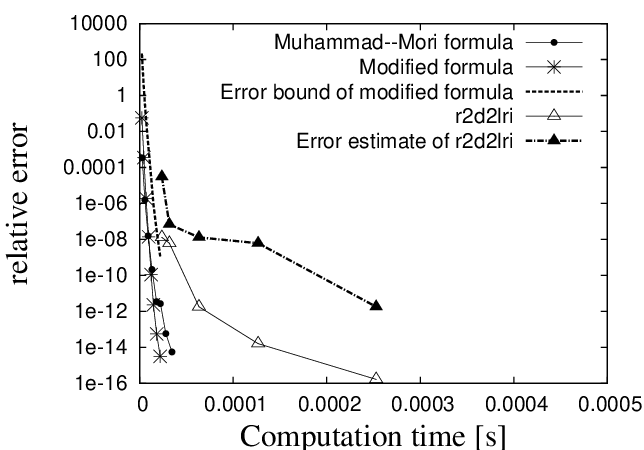}
\caption{Relative error with respect to computation time in Example~\ref{exam:deriv-singular}.}
\label{fig:deriv-singular-t}
\end{minipage}
\end{center}

\begin{center}
\begin{minipage}{0.45\linewidth}
\includegraphics[width=\linewidth]{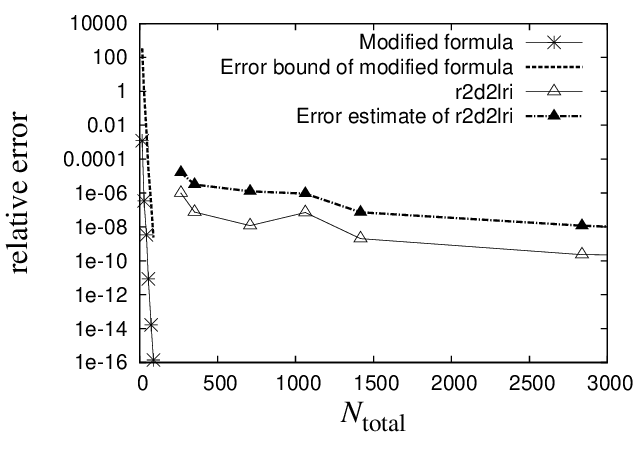}
\caption{Relative error with respect to $N_{\text{total}}$ in Example~\ref{exam:weak-singular}.}
\label{fig:weak-singular-N}
\end{minipage}
\begin{minipage}{0.05\linewidth}
\mbox{ }
\end{minipage}
\begin{minipage}{0.45\linewidth}
\includegraphics[width=\linewidth]{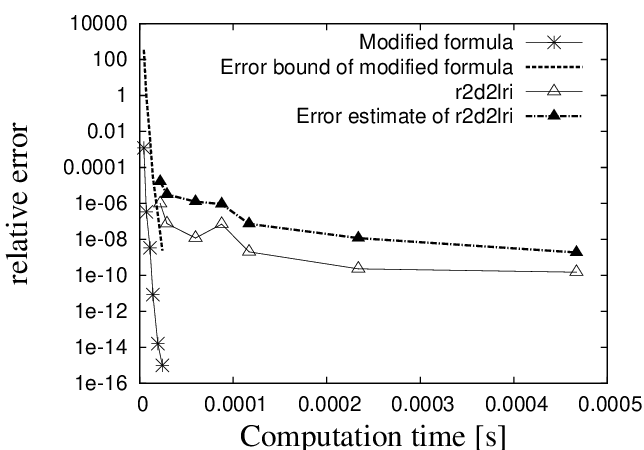}
\caption{Relative error with respect to computation time in Example~\ref{exam:weak-singular}.}
\label{fig:weak-singular-t}
\end{minipage}
\end{center}
\end{figure}


\section{Proofs}
\label{sec:proofs}

In this section,
only the inequality~\eqref{leq:main-result-error-bound}
(for $|I - I_{\textDE}^{\text{\rm{inc}}}(h)|$)
is proved, since
$|I - I_{\textDE}^{\text{\rm{dec}}}(h)|$
is bounded in exactly the same way.
Let us have a look at the sketch of the proof first.

\subsection{Sketch of the proof}

The error $|I - I_{\textDE}^{\text{\rm{inc}}}(h)|$
can be bounded by a sum of two terms as follows:
\begin{align*}
|I - I_{\textDE}^{\text{\rm{inc}}}(h)|
&\leq
\left|
\int_a^b F(x)\diff x
-\tilde{h}\sum_{i=-M_{-}}^{M_{+}} F(\DEt(i\tilde{h}))\DEtDiv(i\tilde{h})
\right|\\
&\quad +
\tilde{h}\sum_{i=-M_{-}}^{M_{+}}\DEtDiv(i\tilde{h})
\left|
\int_a^{\DEt(i\tilde{h})} f_i(s)\diff s
-
\sum_{j=-N_{-}}^{N_{+}} f_i(\DEt(jh))\DEtDiv(jh)J(j,h)(i\tilde{h})
\right|,
\end{align*}
where $F(x)=\int_a^x f(x,q(s))q'(s)\diff s$,
$f_i(s)=f(\DEt(i\tilde{h}), q(s))q'(s)$,
and $\tilde{h}=2h$.
The first term (say $E_1$) and the second term (say $E_2$)
are bounded as follows:
\begin{align}
E_1
&\leq
\frac{\Bfunc(\gamma,\delta)c_{\gamma,\delta,d}}{\mu}
\left\{
\rme^{\frac{\pi}{2}\overline{\mu}}
+\frac{2c_{\alpha,\beta,d}}{1-\rme^{-2\pi d/\tilde{h}}}
\right\}
2K (b-a)^{\alpha+\beta+\gamma+\delta-2} \rme^{-2\pi d/\tilde{h}},
\label{leq:bound-E1}\\
E_2
&\leq
\frac{1}{\nu}
\left\{\Bfunc(\alpha,\beta)+\frac{4 c_{\alpha,\beta,d}}{\mu}
\frac{\rme^{-2\pi d/\tilde{h}}}{1-\rme^{-2\pi d/\tilde{h}}}
\right\}
\left\{
1.1\rme^{\frac{\pi}{2}\overline{\nu}}
+\frac{hc_{\gamma,\delta,d}}{d(1-\rme^{-2\pi d/h})}
\right\}
2K (b-a)^{\alpha+\beta+\gamma+\delta-2} \rme^{-\pi d/h}.
\label{leq:bound-E2}
\end{align}
Then, taking $\tilde{h}=2h$, we get the
desired inequality~\eqref{leq:main-result-error-bound}.
In what follows,
the inequalities~\eqref{leq:bound-E1}
and~\eqref{leq:bound-E2} are shown
in Sections~\ref{subsec:bound-E1}
and~\ref{subsec:bound-E2}, respectively.

\subsection{Bound of $E_1$ (error of the DE-Sinc quadrature)}
\label{subsec:bound-E1}

The following two lemmas are important results for this project.

\begin{lem}[Okayama et al.~{\cite[Lemma~4.16]{okayama09:_error}}]
\label{lem:DE-quad-discrete-error}
Let $\tilde{L}$, $\alpha$, and $\beta$ be positive constants,
and let $\mu=\min\{\alpha,\,\beta\}$.
Let $F$ be analytic on $\DEt(\domD_d)$ for $d$ with $0<d<\pi/2$, and satisfy
\begin{equation*}
|F(z)|\leq \tilde{L} |z-a|^{\alpha-1}|b-z|^{\beta-1}
\end{equation*}
for all $z\in\DEt(\domD_d)$.
Then it holds that
\[
\left|
 \int_a^b F(x)\diff x
- \tilde{h}\sum_{i=-\infty}^{\infty}F(\DEt(i\tilde{h}))\DEtDiv(i\tilde{h})
\right|
\leq \tilde{C}_1 \tilde{C}_2 \frac{\rme^{-2\pi d/\tilde{h}}}{1 - \rme^{-2\pi d/\tilde{h}}},
\]
where the constants $\tilde{C}_1$ and $\tilde{C}_2$ are defined by
\begin{align}
\tilde{C}_1=\frac{2\tilde{L}(b-a)^{\alpha+\beta-1}}{\mu},
\quad \tilde{C}_2=2 c_{\alpha,\beta,d}.
\label{eq:def-tilde-C1-C2}
\end{align}
\end{lem}
\begin{lem}[Okayama et al.~{\cite[Lemma~4.18]{okayama09:_error}}]
\label{lem:DE-quad-truncate-error}
Let the assumptions in Lemma~\ref{lem:DE-quad-discrete-error}
be fulfilled.
Furthermore, let $\overline{\mu}=\max\{\alpha,\,\beta\}$,
let $m$ be a positive integer,
let $M_{-}$ and $M_{+}$ be positive integers defined
by~\eqref{DE-Sinc-func-approx-M-def},
%
and let $m$ be taken sufficiently large so that
$M_{-}\tilde{h}\geq \rho_{\alpha}$ and $M_{+}\tilde{h}\geq \rho_{\beta}$ hold.
Then it holds that
\[
\left|
  \tilde{h}\sum_{i=-\infty}^{-(M_{-}+1)}F(\DEt(i\tilde{h}))\DEtDiv(i\tilde{h})
+ \tilde{h}\sum_{i=M_{+}+1}^{\infty}F(\DEt(i\tilde{h}))\DEtDiv(i\tilde{h})
\right|
\leq \rme^{\frac{\pi}{2}\overline{\mu}}
\tilde{C}_1 \rme^{-\frac{\pi}{2}\mu\exp(m\tilde{h})},
\]
where $\tilde{C}_1$ is a constant defined in~\eqref{eq:def-tilde-C1-C2}.
\end{lem}

What should be checked here is whether
the conditions of those two lemmas are satisfied under the
assumptions in Theorem~\ref{thm:main-result}.
The next lemma answers to this question.

\begin{lem}
\label{lem:essential-bound-DE-Sinc}
Let the assumptions in Theorem~\ref{thm:main-result}
be fulfilled,
and let $F$ be defined as
$F(z) = \int_a^{z} f(z, q(w))q'(w)\diff w$.
Then,
the assumptions of Lemmas~\ref{lem:DE-quad-discrete-error}
and~\ref{lem:DE-quad-truncate-error}
are satisfied
with
$\tilde{L}=K(b-a)^{\gamma+\delta-1}\Bfunc(\gamma,\delta) c_{\gamma,\delta,d}$.
%
\end{lem}

If this lemma is proved, combining
Lemmas~\ref{lem:DE-quad-discrete-error}
and~\ref{lem:DE-quad-truncate-error},
and using the
relations~\eqref{eq:m-and-n-def}--\eqref{DE-Sinc-func-approx-N-def},
we get the desired inequality~\eqref{leq:bound-E1}.
For the proof of Lemma~\ref{lem:essential-bound-DE-Sinc},
we need the following inequalities.

\begin{lem}[Okayama et al.~{\cite[Lemma~4.22]{okayama09:_error}}]
\label{Lem:Bound-DEt-essential}
Let $x$ and $y$ be real numbers with $|y|<\pi/2$. Then we have
\begin{align*}
     \left|
      \frac{1}{1+\rme^{\pi\sinh(x+\imnum y)}}
     \right|
&\leq\frac{1}{(1+\rme^{\pi\sinh(x)\cos y})\cos(\frac{\pi}{2}\sin y)},
\\
     \left|
      \frac{1}{1+\rme^{-\pi\sinh (x+\imnum y)}}
     \right|
&\leq\frac{1}{(1+\rme^{-\pi\sinh(x)\cos y})\cos(\frac{\pi}{2}\sin y)}.
\end{align*}
\end{lem}
\begin{lem}
\label{lem:DEt-Complex-Betafunc-Bound}
Let $x,\,\xi,\,y\in\mathbb{R}$ with $|y|<\pi/2$,
let $\gamma$ and $\delta$ be positive constants,
and let us define a function $\DEtZero(x,y)$ as
\begin{equation*}
 \DEtZero(x,y) = \frac{1}{2}\tanh\left(\frac{\pi\cos y}{2}\sinh x\right)
+\frac{1}{2}.
\end{equation*}
Then it holds that
\begin{equation*}
\int_{-\infty}^{\xi}
\frac{\pi|\cosh(x+\imnum y)|\diff x}
     {|1+\rme^{-\pi\sinh(x+\imnum y)}|^{\gamma}|1+\rme^{\pi\sinh(x+\imnum y)}|^{\delta}}
\leq \frac{\Bfunc(\DEtZero(\xi,y);\gamma,\delta)}
          {\cos^{\gamma+\delta}(\frac{\pi}{2}\sin y)\cos y},
\end{equation*}
where $\Bfunc(t; \kappa, \lambda)$ is the incomplete beta function.
\end{lem}
\begin{proof}
From Lemma~\ref{Lem:Bound-DEt-essential}
and $|\cosh(x+\imnum y)|\leq \cosh(x)$, we obtain
\begin{align*}
&\int_{-\infty}^{\xi}
\frac{\pi|\cosh(x+\imnum y)|\diff x}
     {|1+\rme^{-\pi\sinh(x+\imnum y)}|^{\gamma}|1+\rme^{\pi\sinh(x+\imnum y)}|^{\delta}}\\
&\leq
\frac{1}{\cos^{\gamma+\delta}(\frac{\pi}{2}\sin y)\cos y}
\int_{-\infty}^{\xi}
\frac{\pi\cosh(x)\cos (y)\diff x}
     {(1+\rme^{-\pi\sinh(x)\cos y})^{\gamma}(1+\rme^{\pi\sinh(x)\cos y})^{\delta}}
=\frac{\Bfunc(\DEtZero(\xi,y);\gamma,\delta)}
          {\cos^{\gamma+\delta}(\frac{\pi}{2}\sin y)\cos y}.
\end{align*}
\vskip-\baselineskip\vskip-0.5\baselineskip \qed
\end{proof}

By using the estimates,
Lemma~\ref{lem:essential-bound-DE-Sinc}
is proved as follows.

\begin{proof}
The estimate of the constant $\tilde{L}$ is essential.
Let $\xi=\Re[\DEtInv(z)]$ and $y=\Im[\DEtInv(z)]$,
i.e., $z=\DEt(\xi + \imnum y)$.
By applying $w=\DEt(x + \imnum y)$, we have
\begin{align*}
|F(z)|
&=\left|
\int_{-\infty}^{\xi}
f(z,q(\DEt(x+\imnum y)))q'(\DEt(x+\imnum y))\DEtDiv(x+\imnum y)\diff x
\right|\\
&\leq K|z-a|^{\alpha-1}|b-z|^{\beta-1}
\int_{-\infty}^{\xi} |\DEt(x+\imnum y)- a|^{\gamma-1}
|b-\DEt(x+\imnum y)|^{\delta-1}|\DEtDiv(x+\imnum y)|\diff x\\
&=K|z-a|^{\alpha-1}|b-z|^{\beta-1} (b-a)^{\gamma+\delta-1}
\int_{-\infty}^{\xi}
\frac{\pi|\cosh(x+\imnum y)|\diff x}
     {|1+\rme^{-\pi\sinh(x+\imnum y)}|^{\gamma}|1+\rme^{\pi\sinh(x+\imnum y)}|^{\delta}}.
\end{align*}
Then, the desired bound of $\tilde{L}$ is obtained
by using Lemma~\ref{lem:DEt-Complex-Betafunc-Bound}
and $\Bfunc(\DEtZero(\xi,y);\gamma,\delta)\leq\Bfunc(\gamma,\delta)$.
\qed
\end{proof}

\subsection{Bound of $E_2$ (error of the DE-Sinc indefinite integration)}
\label{subsec:bound-E2}

The following two lemmas are important results for this project.

\begin{lem}[Okayama et al.~{\cite[Lemma~4.19]{okayama09:_error}}]
\label{lem:DE-indef-discrete-error}
Let $L$, $\gamma$, and $\delta$ be positive constants,
and let $\nu=\min\{\gamma,\,\delta\}$.
Let $f$ be analytic on $\DEt(\domD_d)$ for $d$ with $0<d<\pi/2$, and satisfy
\begin{equation*}
|f(w)|\leq L |w-a|^{\gamma-1}|b-w|^{\delta-1}
\end{equation*}
for all $w\in\DEt(\domD_d)$.
Then it holds that
\[
\sup_{x\in(a,\,b)}
\left|
 \int_a^x f(s)\diff s
- \sum_{j=-\infty}^{\infty}f(\DEt(jh))\DEtDiv(jh)J(j,h)(\DEtInv(x))
\right|
\leq \frac{C_1 C_2}{2d}
 \frac{h\rme^{-\pi d/h}}{1 - \rme^{-2\pi d/h}},
\]
where the constants $C_1$ and $C_2$ are defined by
\begin{align}
C_1=\frac{2L(b-a)^{\gamma+\delta-1}}{\nu},
\quad C_2=2 c_{\gamma,\delta,d}.
\label{eq:def-C1-C2}
\end{align}
\end{lem}
\begin{lem}[Okayama et al.~{\cite[Lemma~4.20]{okayama09:_error}}]
\label{lem:DE-indef-truncate-error}
Let the assumptions in Lemma~\ref{lem:DE-indef-discrete-error}
be fulfilled.
Furthermore, let $\overline{\nu}=\max\{\gamma,\,\delta\}$,
let $n$ be a positive integer,
let $N_{-}$ and $N_{+}$ be positive integers defined
by~\eqref{DE-Sinc-func-approx-N-def},
and let $n$ be taken sufficiently large so that
$N_{-}h\geq \rho_{\gamma}$ and $N_{+}h\geq \rho_{\beta}$ hold.
Then it holds that
\begin{align*}
&\sup_{x\in(a,\,b)}\left|
  \sum_{j=-\infty}^{-(N_{-}+1)}f(\DEt(jh))\DEtDiv(jh)J(j,h)(\DEtInv(x))
+ \sum_{j=N_{+}+1}^{\infty}f(\DEt(jh))\DEtDiv(jh)J(j,h)(\DEtInv(x))
\right|\\
&\leq 1.1\rme^{\frac{\pi}{2}\overline{\nu}}
C_1 \rme^{-\frac{\pi}{2}\nu\exp(nh)},
\end{align*}
where $C_1$ is a constant defined in~\eqref{eq:def-C1-C2}.
\end{lem}

What should be checked here is whether
the conditions of those two lemmas are satisfied under the
assumptions in Theorem~\ref{thm:main-result}.
The next lemma answers this question.

\begin{lem}
\label{lem:essential-bound-DE-Sinc-indef}
Let the assumptions in Theorem~\ref{thm:main-result}
be fulfilled,
and let $f_i(z)$ be defined as
$f_i(z) = f(\DEt(i\tilde{h}),q(z))q'(z)$.
Then,
the assumptions of Lemmas~\ref{lem:DE-indef-discrete-error}
and~\ref{lem:DE-indef-truncate-error}
are satisfied with
$f=f_i$ and
$L=K(\DEt(i\tilde{h})-a)^{\alpha-1}(b-\DEt(i\tilde{h}))^{\beta-1}$.
\end{lem}

The proof is omitted
since it is obvious from~\eqref{leq:bound-f-q}.
Combining
Lemmas~\ref{lem:DE-indef-discrete-error}
and~\ref{lem:DE-indef-truncate-error},
and using the
relations~\eqref{eq:m-and-n-def}--\eqref{DE-Sinc-func-approx-N-def},
we have
\begin{align*}
E_2 &\leq
\left[ \tilde{h}\sum_{i=-M_{-}}^{M_{+}} \DEtDiv(i\tilde{h})
(\DEt(i\tilde{h})-a)^{\alpha-1}(b-\DEt(i\tilde{h}))^{\beta-1}\right]\\
&\quad\times
\frac{2K (b-a)^{\gamma+\delta-1}}{\nu}
\left\{1.1\rme^{\frac{\pi}{2}\overline{\nu}}
+\frac{hc_{\gamma,\delta,d}}{d(1-\rme^{-2\pi d /h})}
\right\}
\rme^{-\pi d /h}.
\end{align*}
What is left is to bound the term in $[\,\cdot\,]$,
which is done by the next lemma.

\begin{lem}
Let $\alpha$ and $\beta$ be positive constants,
and let $\mu=\min\{\alpha,\,\beta\}$. Then it holds that
\[
 \tilde{h}\sum_{i=-M_{-}}^{M_{+}}
(\DEt(i\tilde{h})-a)^{\alpha-1}(b-\DEt(i\tilde{h}))^{\beta-1}
\DEtDiv(i\tilde{h})
\leq (b-a)^{\alpha+\beta-1}\left\{
\Bfunc(\alpha,\beta)+
\frac{4 c_{\alpha,\beta,d}}{\mu}
\frac{\rme^{-2\pi d/\tilde{h}}}{1-\rme^{-2\pi d/\tilde{h}}}
\right\}.
\]
\end{lem}
\begin{proof}
Let us define $F$ as $F(x)=(x-a)^{\alpha-1}(b-x)^{\beta-1}$.
We readily see
\begin{align*}
 \tilde{h}\sum_{i=-M_{-}}^{M_{+}} F(\DEt(i\tilde{h}))\DEtDiv(i\tilde{h})
&\leq\tilde{h}\sum_{i=-\infty}^{\infty}
F(\DEt(i\tilde{h}))\DEtDiv(i\tilde{h})\\
&\leq \int_{a}^{b}F(x)\diff x
+ \left|
\int_{a}^{b}F(x)\diff x
-\tilde{h}\sum_{i=-\infty}^{\infty}
F(\DEt(i\tilde{h}))\DEtDiv(i\tilde{h})
\right|,
\end{align*}
and we further see
$\int_a^bF(x)\diff x = (b-a)^{\alpha+\beta-1}\Bfunc(\alpha,\beta)$.
For the second term, use Lemma~\ref{lem:DE-quad-discrete-error}
to obtain
\begin{align*}
\left|
\int_{a}^{b}F(x)\diff x
-\tilde{h}\sum_{i=-\infty}^{\infty}
F(\DEt(i\tilde{h}))\DEtDiv(i\tilde{h})
\right|
\leq \frac{4(b-a)^{\alpha+\beta-1}c_{\alpha,\beta,d}}{\mu}
\frac{\rme^{-2\pi d/\tilde{h}}}{1 - \rme^{-2\pi d/\tilde{h}}},
\end{align*}
which completes the proof.\qed
\end{proof}

\section{Concluding remarks}
\label{sec:conclusion}

Muhammad--Mori~\cite{muhammad05:_iterat}
proposed an approximation formula for~\eqref{eq:target-repeat-int},
which can converge
\emph{exponentially}
with respect to $N_{\text{total}}$
even if $f(x,y)$ or $q(x)$ has boundary singularity.
It is particularly worth noting that
their formula is quite efficient if $f$ is of a product type:
$f(x,y)=X(x)Y(y)$.
However, its convergence was not proved in a precise sense,
and it cannot be used in the case $q'(x)\leq 0$
(only the case $q'(x)\geq 0$ was considered).
This paper improved the formula in the sense that
both cases ($q'(x)\geq 0$ and $q'(x)\leq 0$) are taken into account,
and it can achieve a better convergence rate.
Furthermore, its rigorous error bound that is \emph{computable}
is given, which enables us to guarantee the accuracy
of the approximation mathematically.
Numerical results in Section~\ref{sec:numer_exam}
confirm the error bound and the exponential rate of convergence,
and also suggest that the modified formula
works incredibly accurate if $f$ is of a product type,
similar to the original formula.
This is because,
instead of a \emph{definite} integration formula (quadrature rule),
an \emph{indefinite} integration formula
is employed for the approximation of the inner integral.

However,
as said in the original paper~\cite{muhammad05:_iterat},
the use of the \emph{indefinite} integration formula
has a drawback: it cannot be used when
$f(x,y)$ have singularity along $y=q(x)$, e.g.,
\[
 \int_a^b \left(\int_A^{q(x)}\frac{\diff y}{\sqrt{q(x) - y}}\right),
\quad
 \int_a^b \left(\int_A^{q(x)}\sqrt{(q(x) - y)(q(x)+y)}\diff y\right),
\]
and so on
($f$ can have singularity at the endpoints $y=A$ and $y=B$, though).
This is because the assumption of Theorem~\ref{thm:main-result}
(more precisely, Lemmas~\ref{lem:DE-indef-discrete-error}
and~\ref{lem:DE-indef-truncate-error})
is not satisfied in this case.
In such a case, a \emph{definite} integration formula
should be employed for the approximation of the inner integral.
Actually,
such an approach was already successfully taken in some
one-dimensional cases~\cite{okayama0x:_approx,okayama0x:_sinc}.
It also may work for~\eqref{eq:target-repeat-int},
which will be considered in a future report.





\begin{small}
\bibliographystyle{model1b-num-names}
\bibliography{ErrorEstimRepeatInt}
\end{small}






\end{document}